\documentclass[12pt]{amsart}
\usepackage{amssymb}
\usepackage{amsmath}
\usepackage[all]{xy}
\usepackage{graphicx}
\usepackage[dvips]{epsfig}
\newtheorem{theorem}{Theorem}[section]
\newtheorem{prop}{Proposition}[section]
\newtheorem{cor}{Corollary}[section]
\newtheorem{lem}{Lemma}[section]
\numberwithin{equation}{section}

\begin{document}
\def\a{\alpha}
\def\b{\beta}
\def\om{\omega}
\def\t{\theta}
\def\e{\eta}
\def\vp{\varphi}
\def\l{\lambda}
\def\sg{\sigma}
\def\fl{\longrightarrow}
\def\V{\Vert}
\def\w{\widetilde}
\def\d{\delta}
\def\p{\partial}
\def\fra{\frac}
\def\s{\sum}
\def\n{\nabla}
\def\dif{{\rm d}}
\def\R{{\bf R}}
\def\N{{\bf N}}
\def\Z{{\bf Z}}
\def\C{{\bf C}}
\newtheorem{definition}{Definition}[section]
\newcommand{\dep}[1]{\displaystyle{\frac{\partial}{\partial #1}}}
\newcommand{\deri}[1]{\displaystyle{\frac{{\rm d}}{{\rm d} #1}}}
\newcommand{\dderi}[2]{\displaystyle{\frac{{\rm d} #1}{{\rm d} #2}}}
\newcommand{\grad}{{\rm grad}}
\newcommand{\ddep}[2]{\displaystyle{\frac{\partial #1}{\partial #2}}}
\renewcommand{\thesection}{\arabic{section}}
\renewcommand{\theequation}{\thesection.\arabic{equation}}

\title[A new class of metric $f$-manifolds]{A new class of metric $f$-manifolds}
\thanks{The first and the second authors are partially supported by the project MTM2014-52197-P (MINECO, Spain).}
\author{Pablo Alegre}
\address{Departamento de Econom\'{\i}a, M\'etodos Cuantitativos e Historia de la Econom\'{\i}a, Universidad Pablo de Olavide, Ctra. de Utrera, km. 1, 41013-Sevilla, Spain} \email{psalerue@upo.es}
\author{Luis M. Fern\'andez}
\address{Departamento de Geometr\'{\i}a y Topolog\'{\i}a, Facultad de Matem\'aticas, Universidad de Sevilla, C./ Tarfia, s.n., 41012-Sevilla, Spain.} \email{lmfer@us.es}
\author{Alicia Prieto-Mart\'{\i}n}
\address{Departamento de Geometr\'{\i}a y Topolog\'{\i}a, Facultad de Matem\'aticas, Universidad de Sevilla, C./ Tarfia, s.n., 41012-Sevilla, Spain.} \email{aliciaprieto@us.es}
\begin{abstract}
We introduce a new general class of metric $f$-man\-i\-folds which we call (almost) trans-$S$-manifolds and includes $S$-manifolds, $C$-manifolds, $s$-th Sasakian manifolds and generalized Kenmotsu manifolds studied previously. We prove their main properties and we present many examples which justify their study.
\end{abstract}
\maketitle

\section{Introduction}\label{intro}

In complex geometry, the relationships between the different classes of manifolds can be summarize in the well known diagram by Blair \cite{BLAIRBOOK}:

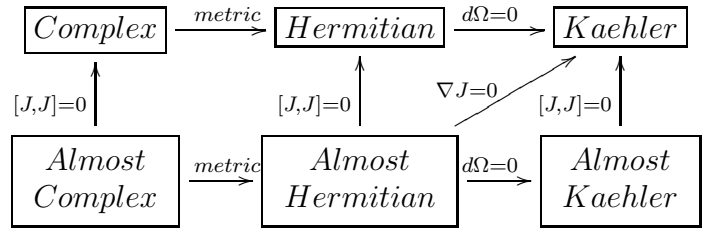
$$\xymatrix{\fbox{$Complex$}\ar"1,2"^{metric} & \fbox{$Hermitian$}\ar"1,3"^{d\Omega=0} & \fbox{$Kaehler$}\\
\fbox{${\begin{array}{c}Almost\\ Complex\end{array}}$}\ar"1,1"^{[J,J]=0} \ar"2,2"^{metric} & \fbox{${\begin{array}{c}Almost \\Hermitian\end{array}}$}\ar"1,2"^{[J,J]=0} \ar"2,3"^{d\Omega=0}\ar"1,3"^{\nabla J=0} & \fbox{${\begin{array}{c}Almost \\ Kaehler\end{array}}$} \ar"1,3"^{[J,J]=0} }$$

And the same for contact geometry:

$$
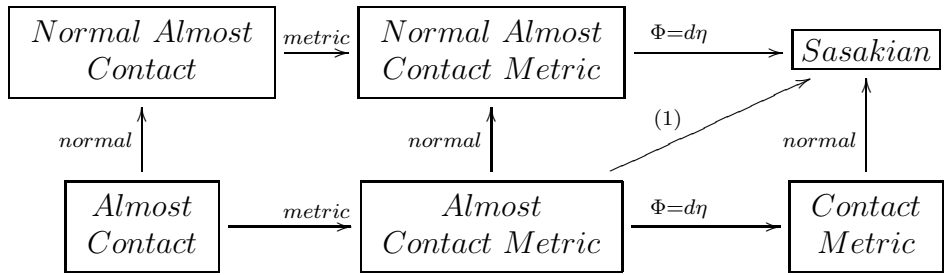
\xymatrix{\fbox{${\begin{array}{c}Normal\mbox{ }Almost\\Contact\end{array}}$}\ar"1,2"^{metric} & \fbox{${\begin{array}{c}Normal\mbox{ }Almost\\Contact\mbox{ }Metric\end{array}}$}\ar"1,4"^{\Phi=d\eta} &  & \fbox{$Sasakian$}\\
\fbox{${\begin{array}{c}Almost\\Contact\end{array}}$}\ar"1,1"^{normal} \ar"2,2"^{metric} & \fbox{${\begin{array}{c}Almost\\Contact\mbox{ }Metric\end{array}}$}\ar"1,2"^{normal} \ar"2,4"^{\Phi=d\eta}\ar"1,4"^{(1)} & & \fbox{${\begin{array}{c}Contact\\Metric\end{array}}$} \ar"1,4"^{normal} }$$

In the above diagram the almost contact structure $(\phi,\eta,\xi)$ is said to be normal if $[\phi,\phi]+2d\eta\otimes\xi=0$ and condition (1) is
$$(\nabla_X \phi)Y=g(X,Y)\xi-\eta(Y)X,$$
for any tangent vector fields $X$ and $Y$.

Moreover, an almost contact metric manifold $(M,\phi,\xi,\eta,g)$ is said to have an $(\alpha,\beta)$ trans-Sasakian structure if (see \cite{oubi} for more details)
\begin{equation}\label{trans}
(\nabla_X\phi)Y \ = \ \alpha\{g(X,Y)\xi-\eta(Y)X\}+\beta\{g(\phi X,Y)\xi-\eta(Y)\phi X\},
\end{equation}
where $\alpha,\beta$ are differentiable functions (called characteristic functions) on $M$. Particular cases of trans-Sasakian manifolds are Sasakian ($\alpha=1,\beta=0$), cosymplectic ($\alpha=\beta=0$) or Kenmotsu ($\alpha=0,\beta=1$) manifolds. In fact, we can extend the above diagram to

$$\xymatrix{\fbox{${\begin{array}{c}Normal\mbox{ }Almost\\Contact\mbox{ }Metric\end{array}}$}\ar"1,2"^{(\ref{deta4})} & \fbox{$Trans-Sasakian$}\\
\fbox{${\begin{array}{c}Almost\\Contact\mbox{ }Metric\end{array}}$}\ar"1,1"^{normal} \ar"2,2"^{(\ref{deta3})}\ar"1,2"^{(\ref{trans})} & \fbox{${\begin{array}{c}Almost\\Trans-Sasakian\end{array}}$}\ar"1,2"^{normal}}$$

\noindent where
\begin{equation}\label{deta3}
d\Phi=\Phi\wedge(\phi^*(\delta\Phi)-(\delta\eta)\eta),\quad d\eta=\frac{1}{2n}\{\delta\Phi(\xi)\Phi-2\eta\wedge\phi^*(\delta\Phi)\}.
\end{equation}
and:
\begin{equation}\label{deta4}
d\Phi=\frac{-1}{n}\delta\eta(\Phi\wedge\eta), \quad d\eta=\frac{1}{2n}\delta\Phi(\xi)\Phi, \quad \phi^*(\delta\Phi)=0.
\end{equation}

More in general, K. Yano \cite{Y} introduced the notion of $f$-structure on a $(2n+s)$-dimensional manifold as a tensor field $f$ of type (1,1) and rank $2n$ satisfying $f^3+f=0$. Almost complex ($s=0$) and almost contact ($s=1$) structures are well-known examples of $f$-structures. In this context, D.E. Blair \cite{B} defined $K$-manifolds (and particular cases of $S$-manifolds and $C$-manifolds). Then, $K$-manifolds are the analogue of Kaehlerian manifolds in the almost complex geometry and $S$-manifolds (resp., $C$-manifolds) of Sasakian manifolds (resp., cosymplectic manifolds) in the almost contact geometry. Consequently, one can obtain a similar diagram for metric $f$-manifolds, that is, manifolds endowed with an $f$-structure and a compatible metric.

The purpose of the present paper is to introduce a new class of metric $f$-manifolds which generalizes that one of trans-Sasakian manifolds. In this context, we notice that there has been a previous generalization of $(\alpha,0)$-trans-Sasakian manifolds for metric $f$-manifolds. It was due to I. Hasegawa, Y. Okuyama and T. Abe who introduced the so-called {\it homothetic $s$-contact Riemannian manifolds} in \cite{hoa} as metric $f$-manifolds such that $2c_ig(fX,Y)=d\eta_i(X,Y)$ for certain nonzero constants $c_i$, $i=1,\dots,s$ (actually, they use $p$ instead of $s$). In particular, if the structure vector fields $\xi_i$ are Killing vector fields and the $f$-structure is also normal, the manifold is called a  {\it homothetic $s$-th Sasakian manifold}. They proved that a homothetic $s$-contact Riemannian manifold is a homothetic $s$-th Sasakian manifold if and only if
\begin{equation*}\label{nabla1}
(\nabla_X f)Y \ = -\sum_{i=1}^sc_i\{g(fX,fY)\xi_i+\eta_i(Y)f^2X\},
\end{equation*} and
\begin{equation*}\label{nablaxi1}
\nabla_X\xi_i=c_ifX,
\end{equation*}
for any tangent vector fields $X$ and $Y$ and any $i=1,\dots,s$.

More recently, M. Falcitelli and A.M. Pastore have introduced $f$-structures of Kenmotsu type as those normal $f$-manifolds with $dF=2\eta^1\wedge F$ and $d\eta^i=0$ for $i=1,\dots,s$ \cite{falcitelli}. In this context, L. Bhatt and K.K. Dube \cite{BD} and A. Turgut Vanli and R. Sari \cite{TS} have studied a more general type of Kenmotsu $f$-manifolds for which all the structure 1-forms $\eta_i$ are closed and:
$$dF=2\displaystyle{\sum_{i=1}^s}\eta_i\wedge F.$$

These examples justify the idea of introducing the mentioned new more general class of metric $f$-manifolds, including the above ones, which we shall call {\it trans-$S$-manifolds} because trans-Sasakian manifolds become to be a particular case of them.

The paper is organized as follows: after a preliminaries section concerning metric $f$-manifolds, in Section \ref{sec3} we define almost trans-$S$-manifolds and trans-$S$-manifolds in terms of the derivative of the $f$-structure and some characteristic functions and study their main properties. Specially, we prove a characterization theorem which gives a necessary and sufficient condition for an almost trans-$S$-manifold to be a trans-$S$-manifold, that is, for the normality of the structure, concerning the derivative of the structure vector fields in any direction. Moreover, we observe that $S$-manifolds, $C$-manifolds and Kenmotsu $f$-manifolds actually are trans-$S$-manifolds. On the other hand, we get some desirable conditions to be satisfied for trans-$S$-manifolds in order to generalize those ones of trans-Sasakian manifolds. By using them, we characterize what trans-$S$-manifolds are $K$-manifolds and we justify that these two classes of metric $f$-manifolds are not related by inclusion.

Finally, in the last section, we present many non-trivial examples, that is, with non-constant characteristic functions, of (almost) trans-$S$-manifolds. To this end, we use generalized $D$-conformal deformations and warped products as tools.

{\it Acknowledgement}: The first and the second authors are partially supported by the project MTM2014-52197-P (MINECO, Spain).

\section{Metric $f$-manifolds}\label{fman}

A $(2n+s)$-dimensional Riemannian manifold $(M,g)$ endowed with an $f$-structure $f$ (that is, a tensor field of type (1,1) and rank $2n$ satisfying $f^3+f=0$ \cite{Y}) is said to be a {\it metric $f$-manifold} if, moreover, there exist $s$ global vector fields $\xi_1,\dots ,\xi_s$ on $M$ (called {\it structure vector fields}) such that, if $\e_1,\dots ,\e_s$ are the dual 1-forms of $\xi_1,\dots ,\xi_s$, then $$f\xi_\a=0;\mbox{ }\e_\a\circ f=0;\mbox{ }f^2=-I+\s_{\a=1}^s\e_\a\otimes\xi_\a;$$
\begin{equation} g(X,Y)=g(fX,fY)+\s_{i=1}^s\e_i(X)\e_i(Y),\end{equation}
for any $X,Y\in \mathcal{X}(M)$ and $i=1,\dots,s$. The distribution on $M$ spanned by the structure vector fields is denoted by $\mathcal{M}$ and its complementary orthogonal distribution is denoted by $\mathcal{L}$. Consequently, $TM=\mathcal{L}\oplus\mathcal{M}$. Moreover, if $X\in\mathcal{L}$, then $\e_\a(X)=0$, for any $\a=1,\dots ,s$ and if $X\in\mathcal{M}$, then $fX=0$.

For a metric $f$-manifold $M$ we can construct very useful local orthonormal basis of tangent vector fields. To this end, let $U$ be a coordinate neighborhood on $M$ and $X_1$ any unit vector field on $U$, orthogonal to the structure vector fields. Then, $fX_1$ is another unit vector field orthogonal to $X_1$ and to the structure vector fields too. Now, if it is possible, we choose a unit vector field $X_2$ orthogonal to the structure vector fields, to $X_1$ and to $fX_1$. Then, $fX_2$ is also a unit vector field orthogonal to the structure vector fields, to $X_1$, to $fX_1$ and to $X_2$. Proceeding in this way, we obtain a local orthonormal basis $\{X_i,fX_i,\xi_j\}$, $i=1,\dots,n$ and $j=1,\dots,s$, called an {\it $f$-basis}.

Let $F$ be the 2-form on $M$ defined by $F(X,Y)=g(X,fY)$, for any $X,Y\in\mathcal{X}(M)$. Since $f$ is of rank $2n$, then
$$\eta_1\wedge\cdots\wedge\eta_s\wedge F^n\neq 0$$
and, in particular, $M$ is orientable. A metric $f$-manifold is said to be a {\it metric $f$-contact manifold} if $F=\dif\e_i$, for any $i=1,\dots,s$.

The $f$-structure $f$ is said to be {\it normal} if
$$[f,f]+2\sum_{i=1}^s\xi_i\otimes d\eta_i=0,$$
where $[f,f]$ denotes the Nijenhuis tensor of $f$. If $f$ is normal, then \cite{GY}
\begin{equation}\label{normalbracket}
[\xi_i,\xi_j]=0,
\end{equation}
for any $i,j=1,\dots,s$.

A metric $f$-manifold is said to be a $K$-manifold \cite{B} if it is normal and $\dif F=0$. In a $K$-manifold $M$, the structure vector fields are Killing vector fields \cite{B}. A $K$-manifold is called an $S$-manifold if $F=\dif\e_i$, for any $i$ and a $C$-manifold if $\dif\e_i=0$, for any $i$. Note that, for $s=0$, a $K$-manifold is a Kaehlerian manifold and, for $s=1$, a $K$-manifold is a quasi-Sasakian manifold, an $S$-manifold is a Sasakian manifold and a $C$-manifold is a cosymplectic manifold. When $s\geq 2$, non-trivial examples can be found in \cite{B,hoa}. Moreover, a $K$-manifold $M$ is an $S$-manifold if and only if
\begin{equation}\label{nablaXinS}
\n_X\xi_i=-fX,\mbox{ }X\in\mathcal{X}(M),\mbox{ }i=1,\dots,s,
\end{equation}
and it is a $C$-manifold if and only if
\begin{equation}\label{nablaXinC}
\n_X\xi_i=0,\mbox{ }X\in\mathcal{X}(M),\mbox{ }i=1,\dots,s.
\end{equation}

It is easy to show that in an $S$-manifold,
\begin{equation}\label{nablaXf}
(\n_Xf)Y=\sum_{i=1}^s\left\{g(fX,fY)\xi_i+\e_i(Y)f^2X\right\},
\end{equation}
for any $X,Y\in\mathcal{X}(M)$ and in a $C$-manifold,
\begin{equation}\label{nablaXfinC}
\nabla f=0.
\end{equation}

\section{Definition of trans-S-manifolds and main properties} \label{sec3}
The original idea to define $(\alpha,\beta)$ trans-Sasakian manifolds is to generalize cosymplectic, Kenmotsu and Sasakian manifolds.
\begin{center}
\renewcommand
{\arraystretch}{1.5}
%\par \bigskip
%\setlength{\extrarowheight}{2pt}
\begin{tabular}{|c | c| c|}
\hline
\begin{tabular}{c}
$\begin{array}{c}
\mbox{Kenmotsu:}\\
d\eta=0, normal\\
\hline
\mbox{Cosymplectic:}\\
 d\Phi=0,\quad d\eta=0, \\
  \mbox{normal}\\
 \hline
\mbox{Sasakian:}\\
\Phi=d\eta, \mbox{ normal}
\end{array}$
\end{tabular} &
$\begin{array}{c}
\mbox{Quasi-Sasakian:}\\
d\Phi=0, \\
\mbox{normal}
\end{array}$ &
$\begin{array}{c}
\mbox{Trans-Sasakian:}\\
d\Phi=2\beta(\Phi\wedge\eta),\\
d\eta=\alpha\Phi,\\
\phi^*(\delta\Phi)=0,\\
\mbox{normal}
\end{array}$\\
\hline
\end{tabular}
\end{center}

\bigskip
In the same way, our idea is to define {\it trans-$S$-manifolds} generalizing $C$-manifolds, $f$-Kenmotsu and $S$-manifolds.

As we said in the Introduction, an almost contact manifold is trans-Sasakian if and only if it (\ref{trans}) holds. This property aims us to introduce trans-$S$-manifolds.

\begin{definition}
A $(2n+s)$-dimensional metric $f$-manifold $M$ is said to be a {\bf almost trans-$S$-manifold} if it satisfies \begin{equation}\label{nabla}
\begin{split}
(\nabla_Xf)Y=\sum_{i=1}^s&\left[\alpha_i\{g(fX,fY)\xi_i+\eta_i(Y)f^2X\}\right.\\
&+\left.\beta_i\{g(fX,Y)\xi_i-\eta_i(Y)fX\}\right],
\end{split}
\end{equation}
for certain smooth functions $($called the {\bf characteristic functions}$)$ $\alpha_i,\beta_i$, $i=1....s$, on $M$ and any $X,Y\in\mathcal{X}(M)$. If, moreover, $M$ is normal, then it is said to be a {\bf trans-$S$-manifold}.
\end{definition}

So, if $s=1$, a trans-$S$-manifold is actually a trans-Sasakian manifold. Furthermore, in this case, condition (\ref{nabla}) implies normality. However, for $s\geq 2$, this does not hold. In fact, it is straightforward to prove that, for any $X,Y\in\mathcal{X}(M)$,
\begin{equation}\label{normal}
\begin{split}
[f,f](X,Y)+&2\sum_{i=1}^sd\eta_i(X,Y)\xi_i\\
&=\sum_{i,j=1}^s\left[\eta_j(\nabla_X\xi_i)\eta_j(Y)-\eta_j(\nabla_Y\xi_i)\eta_j(X)\right]\xi_i,
\end{split}
\end{equation}
which is not zero in general. But, in a trans-$S$-manifold, (\ref{normal}) implies that, for any $X\in\mathcal{X}(M)$ and any $i=1,\dots,s$:
$$\sum_{j=1}^s\eta_j(\nabla_X\xi_i)\eta_j(Y)-\sum_{j=1}^s\eta_j(\nabla_Y\xi_i)\eta_j(X)=0.$$

If we put $Y=\xi_k$, from (\ref{normalbracket}) we get that
\begin{equation}\label{etaknablaXxii}
\eta_k(\nabla_X\xi_i)=0,
\end{equation}
for any $i,k=1,\dots,s$. Using this fact, from (\ref{nabla}), we deduce that
\begin{equation}\label{ts2}
\nabla_X\xi_i=-\alpha_i f X-\beta_if^2X,
\end{equation}
for any $X\in\mathcal{X}(M)$ and any $i=1,\dots,s$.

Now, we can prove:
\begin{theorem}\label{teonablaxi}
A almost trans-$S$-manifold $M$ is a trans-$S$-man\-i\-fold if and only if (\ref{ts2}) holds for any $X\in\mathcal{X}(M)$ and any $i=1,\dots,s$.
\end{theorem}
\begin{proof}
From (\ref{nabla}) we have that, for any $X\in\mathcal{X}(M)$ and any $i=1,\dots,s$:
$$\nabla_X\xi_i=-\alpha_i f X-\beta_if^2X+\sum_{j=1}^s\eta_j(\nabla_X \xi_i)\xi_j.$$

Comparing this equality and (\ref{ts2}) we have that $\eta_j(\nabla_X\xi_i)=0$, for any $i,j=1,\dots,s$. So, from (\ref{normal}), the metric $f$-manifold $M$ is normal and, consequently, a trans-$S$-manifold. The converse is obvious.
\end{proof}

Observe that (\ref{nabla}) can be re-written as
\begin{equation*}\label{nablaF}
\begin{split}
(\nabla_XF)(Y,Z)=\sum_{i=1}^s&[\a_i\{g(fX,fZ)\e_i(Y)-g(fX,fY)\e_i(Z)\}\\
&+\b_i\{g(X,fY)\e_i(Z)-g(X,fZ)\e_i(Y)\}],
\end{split}
\end{equation*}
for any $X,Y,Z\in\mathcal{X}(M)$. Then, if $X\in\mathcal{L}$ is a unit vector field, we have:
$$(\nabla_XF)(X,\xi_i)=-\a_i,\mbox{ }(\nabla_XF)(\xi_i,fX)=\b_i,\mbox{ }i=1,\dots,s.$$

Moreover, from (\ref{ts2}), we deduce
\begin{equation*}
(\nabla_X\e_i)Y=\a_ig(X,fY)+\b_ig(fX,fY),
\end{equation*}
for any $X,Y\in\mathcal{X}(M)$ and any $i=1,\dots,s$. Again, if $X\in\mathcal{L}$ is a unit vector field, we get:
$$(\nabla_X\e_i)fX=-\a_i,\mbox{ }(\nabla_X\e_i)X=\b_i,\mbox{ }i=1,\dots,s.$$

For trans-$S$-manifolds, we can prove the following theorem.

\begin{theorem}\label{alfabeta}
Let $M$ be a trans-$S$-manifold. Then, $(\delta F)\xi_i=2n\alpha_i$ and $\delta\eta_i=-2n\beta_i$, for any $i=1,\dots,s$.
\end{theorem}
\begin{proof}
Taking a $f$-basis $\{X_1,\dots,X_n,fX_1,\dots,fX_n,\xi_1,\dots,\xi_s\}$,  since
\begin{equation*}
\begin{split}
(\delta F)X=&-\sum_{k=1}^n\left\{(\nabla_{X_k}F)(X_k,X)+(\nabla_{fX_k}F)(fX_k,X)\right\}\\
&-\sum_{j=1}^s(\nabla_{\xi_j}F)(\xi_j,X)\\
=&\sum_{k=1}^n\left\{g(X_k,(\nabla_{X_k}\phi)X)+ g(fX_k,(\nabla_{fX_k}\phi)X)\right\},
\end{split}
\end{equation*}
for any $X\in\mathcal{X}(M)$, by using (\ref{nabla}) it is straightforward to obtain
\begin{equation}\label{ts3}
(\delta F)X=2n\sum_{j=1}^s\alpha_j\eta_j(X)
\end{equation}
and, putting $X=\xi_i$, it follows that $(\delta F)\xi_i=2n\alpha_i$.

Moreover,
$$\delta\eta_i=-\sum_{k=1}^n\left\{(\nabla_{X_k}\eta_i)X_k+(\nabla_{fX_k}\eta_i)fX_k\right\}-\sum_{j=1}^s
(\nabla_{\xi_j}\eta_i)\xi_j,$$
for any $i=1,\dots,s$. But, from (\ref{etaknablaXxii}) we get that $(\nabla_{\xi_j}\eta_i)\xi_j=0$, for any $j=1,\dots,s$. Consequently, by using (\ref{ts2})
\begin{equation*}
\begin{split}
\delta\eta_i=&-\sum_{k=1}^n\left\{g(X_k,\nabla_{X_k}\xi_i)+g(fX_k,\nabla_{fX_k}\xi_i)\right\}\\
=&-\sum_{k=1}^n\beta_i\left\{g(X_k,X_k)+g(fX_k,fX_k)\right\}=-2n\beta_i.
\end{split}
\end{equation*}
which concludes the proof.
\end{proof}

The above theorem generalizes the result given by D.E. Blair and J.A. Oubi\~{n}a in \cite{BO} for trans-Sasakian manifolds. Moreover, trans-$S$-manifolds verify certain desirable conditions.

\begin{prop}\label{p31}
Let $M$ be a trans-S-manifold. The following equations are verified:
\begin{enumerate}
\item[(i)] $dF=2F\wedge\displaystyle{\sum_{i=1}^s}\beta_i\eta_i$;
\item[(ii)] $d\eta_i=\alpha_iF$, $i=1,\dots,s$;
\item[(iii)] $f^*(\delta F)=0$.
\end{enumerate}
\end{prop}
\begin{proof}
From (\ref{nabla}), a direct computation gives, for any $X,Y,Z\in\mathcal{X}(M)$:
\begin{equation*}
\begin{split}
dF(X,Y,Z)=&-g((\nabla_Xf)Y,Z)+g((\nabla_Yf)X,Z)-g((\nabla_Zf)X,Y)\\
=&2\sum_{i=1}^s\{-\beta_i\eta_i(Z)g(fX,Y)+\beta_i\eta_i(Y)g(fX,Z)\\
&-\beta_i\eta_i(X)g(fY,Z)\}\\
=&2(F\wedge\sum_{i=1}^s\beta_i\eta_i)(X,Y,Z).
\end{split}
\end{equation*}

Next, from (\ref{ts2}) it is obtained the second statement. Finally, from (\ref{ts3}) we get (iii).
\end{proof}

From $(ii)$ of the above proposition we observe that if one of the functions $\a_i$ is a non-zero constant function, then the 2-form $F$ is closed and the trans-$S$-manifold $M$ is a $K$-manifold. Moreover we can prove:
\begin{theorem}
A trans-$S$-manifold $M$ is a $K$-manifold if and only if $\b_1=\cdots=\b_s=0$.
\end{theorem}
\begin{proof}
Firstly, if all the functions $\b_i$ are equal to zero, from $(i)$ of Proposition \ref{p31} we get $dF=0$ and $M$ is a $K$-manifold.

Conversely, it is known (see \cite{F}) that, for $K$-manifolds, the following formula holds, for any $X,Y,Z\in\mathcal{X}(M)$:
$$g((\nabla_Xf)Y,Z)=\sum_{i=1}^s\{d\e_i(fY,X)\e_i(Z)-d\e_i(fZ,X)\e_i(Y)\}.$$

Consequently, from $(ii)$ of Proposition \ref{p31} and (\ref{nabla}) we conclude that $\b_i=0$, for any $i=1,\dots,s$.
\end{proof}

From Theorem \ref{alfabeta} we deduce:
\begin{cor}
A trans-$S$-manifold $M$ is a $K$-manifold if and only if $\delta\e_i=0$, for any $i=1,\dots,s$.
\end{cor}

Furthermore, taking into account (\ref{nablaXinS}) and (\ref{nablaXinC}), we have:
\begin{cor}
Any trans-$S$-manifold is an $S$-manifold if and only if $\a_i=1$, $\b_i=0$ and it is a $C$-manifold if and only if $\a_i=\b_i=0$, in both cases for any $i=1,\dots,s$.
\end{cor}

In next section, we shall present some examples of trans-$S$-manifolds which are not $K$-manifolds due to not all their characteristic functions $\b_i$ are zero. Now, the natural question is if any $K$-manifold is a trans-$S$-manifold. In general, the answer in negative and to this end, we can consider the following example.

Let $(N,J,G)$ be a Kaehler manifold, $(M,f,\xi_1,\dots,\xi_s,\e_1,\dots,\e_s,g)$ be an $S$-manifold and $\w M=N\times M$.

If $\w X=U+X, \w Y=V+Y\in\mathcal{X}(\w M)$, where $U,V\in\mathcal{X}(N)$ and $X,Y\in\mathcal{X}(M)$, respectively, we can define a metric $f$-structure on $\w M$ by the following structure elements:
$$\w f(U+X)=JU+fX,\mbox{ }\w\xi_i=0+\xi_i,\mbox{ }\w\e_i(U+X)\e_i(X),\mbox{ }i=1,\dots,s,$$
$$\w g(U+X,V+Y)=G(U,V)+g(X,Y).$$

It is straightforward to check that $\w M$ with this structure is a $K$-manifold. However, it is not a trans-$S$-manifold. In fact, since $N$ is a Kaehler manifold and so, $J$ is parallel, if $\nabla$ and $\w\nabla$ denote the Riemannian connections of $M$ and $\w M$, respectively, then
$$(\w\n_{\w X}\w f)\w Y=0+(\n_Xf)Y$$
and, consequently, (\ref{nabla}) does not hold for $\w M$.

However, we can observe that, from (\ref{nablaXf}) and (\ref{nablaXfinC}), the particular cases of $S$-manifolds and $C$-manifolds are trans-$S$-manifolds.

On the other hand, it is known \cite{B} that, in a $K$-manifold, all the structure vector fields are Killing vector fields. For trans-$S$-manifolds we can prove:

\begin{prop}
Let $M$ be a trans-$S$-manifold. Then, the structure vector field $\xi_i$ is a Killing vector field if and only if the corresponding characteristic function $\beta_i=0$.
\end{prop}
\begin{proof}
A direct computation by using (\ref{ts2}) gives
$$(L_{\xi_i}g)(X,Y)=2\beta_ig(fX,fY),$$
for any $X,Y\in\mathcal{X}(M)$. This completes the proof.
\end{proof}

\section{Examples of trans-$S$-manifolds}\label{examples}

As we have mentioned above, it is obvious that, from (\ref{nablaXf}) and (\ref{nablaXfinC}), $S$-manifolds and $C$-manifolds are trans-$S$-manifolds. Moreover, the homothetic $s$-th Sasakian manifolds of \cite{hoa} are also trans-$S$-manifolds with the function $\alpha_i$ constant and $\beta_i=0$, for any $i$.

From Propositions 2.2 and 2.5 of \cite{falcitelli}, we see that $f$-manifolds of Kenmotsu type, introduced by M. Falcitelli and A.M. Pastore, actually are trans-$S$-manifolds with functions $\alpha_1=\cdots=\alpha_s=\beta_2=\cdots=\beta_s=0$ and $\beta_1=1$.

Also, from Theorem 2.4 in \cite{TS}, we see that generalized Kenmotsu manifolds studied by L. Bhatt and K.K. Dube \cite{BD} and A. Turgut Vanli and R. Sari \cite{TS} are trans-$S$-manifolds with functions  $\alpha_1=\cdots=\alpha_s=0$ and $\beta_1=\cdots=\beta_s=1$.

Then, we are going to look for examples with different non-constant functions $\alpha_i$ and $\beta_i$. We shall obtain these examples by using $D$-conformal deformations and warped products.

Firstly, given a metric $f$-manifold $(M,f,\xi_1,\dots,\xi_s,\eta_1,\dots,\eta_s,g)$, let
us consider the {\it generalized $D$-conformal deformation} given by
\begin{equation}\label{dconf}
\widetilde f=f, \quad \widetilde\xi_i=\frac{1}{a}\xi_i, \quad \widetilde\eta_i=a\eta_i, \quad
\widetilde g= bg + (a^2-b)\sum_{i=1}^s\eta_i\otimes\eta_i,
\end{equation}
for any $i=1,\dots,s$, where $a,b$ are two positive differentiable functions on $M$. Then, it is easy to
see that $(M,\widetilde f,\widetilde\xi_1,\dots,\widetilde\xi_s,\widetilde\eta_1,\dots,\widetilde\eta_s,\widetilde g)$ is also a metric $f$-manifold. Let us notice that we can obtain conformal, $D$-homothetic (see \cite{tanno}) or $D$-conformal (in the sense of S. Suguri and S. Nakayama \cite{suguri}) deformations, by putting $a^2=b$, $a=b=constant$ or $a=b$ in (\ref{dconf}), respectively. In \cite{olszack} Z. Olszack considered $a$ and $b$ constants, $a\neq 0$, $b>0$ but not necesarily equal and he also called the resulting transformation a $D$-homothetic deformation.

Moreover, let us suppose that $M$ is a trans-$S$-manifold and that $a,b$ depend only on the directions of the structure vector fields $\xi_i$, $i=1,\dots,s$. Therefore, we can calculate $\widetilde\nabla$ from $\nabla$ and $\widetilde g$ by using Koszul's formula and (\ref{ts2}). It follows that the Riemannian connection $\widetilde\nabla$ of $\widetilde g$ is given by
\begin{equation}\label{conexiondconf}
\begin{split}
\widetilde\nabla_XY &= \nabla_XY +\sum_{i=1}^s\frac{2(a^2-b)\beta_i-\xi_ib}{2a^2} g(f X,f Y)\xi_i\\
 &-\frac{1}{2b}\{(Xb)f^2Y+(Yb)f^2X\}\\
 &+\frac{1}{2a^2}\sum_{i=1}^s\left\{(Xa^2)\eta_i(Y)+(Ya^2)\eta_i(X)\right.\\
 &-(\xi_ia^2)\sum_{j=1}^s\left.\eta_j(X)\eta_j(Y)\right\}\xi_i\\
 &-\frac{a^2-b}{b}\sum_{i=1}^s\alpha_i\{\eta_i(Y)fX+\eta_i(X)fY\},\\
\end{split}
\end{equation}
for any vector fields $X,Y\in\mathcal{X}(M)$.
\begin{theorem}\label{dancinginthedark}
Let $(M,f,\xi_1,\dots,\xi_s,\eta_1,\dots,\eta_s,g)$ be a trans-$S$-manifold and consider a ge\-ne\-ra\-lized $D$-conformal deformation on $M$, with $a,b$ positive functions depen\-ding only on the directions of the structure vector fields. Then $(M,\widetilde f,\widetilde\xi_1,\dots,\widetilde\xi_s,\widetilde\eta_1,\dots,\widetilde\eta_s,\widetilde g)$ is also a trans-$S$-manifold with functions:
\begin{equation*}\label{stilllovingyou}
\widetilde\alpha_i=\frac{\alpha_i a}{b},\mbox{  }\widetilde\beta_i=\frac{\xi_ib}{2ab}+\frac{\beta_i}{a},\mbox{  }i=1,\dots,s.
\end{equation*}
\end{theorem}
\begin{proof}
By using (\ref{conexiondconf}) and taking into account that $b$ only depends on the directions of the structure vector fields, we have
\begin{equation*}
\begin{split}
(\widetilde\nabla_X\widetilde f)Y&=(\nabla_Xf)Y-\sum_{i=1}^s\frac{2(a^2-b)\beta_i-\xi_ib}{2a^2}g(fX,Y)\xi_i\\
&-\frac{1}{2b}\sum_{i=1}^s(\xi_ib)\eta_i(Y)fX+\frac{a^2-b}{b}\sum_{i=1}^s\alpha_i\eta_i(Y)f^2X,
\end{split}
\end{equation*}
for any $X,Y\in\mathcal{X}(M)$. Now, since $M$ is trans-S-manifold, from (\ref{nabla}) and (\ref{dconf}) we obtain
\begin{equation*}
\begin{split}
(\widetilde\nabla_X\widetilde f)Y &=\sum_{i=1}^s\{\frac{\alpha_ia}{b}(\widetilde g(\widetilde fX,\widetilde fY)\widetilde\xi_i+\widetilde\eta_i(Y)X)\\
&+\left(\frac{\xi_ib}{2ab}+\frac{\b_i}{a}\right)(\widetilde g(\widetilde fX,Y)\widetilde\xi_i-\widetilde\eta_i(Y)\widetilde fX)\},
\end{split}
\end{equation*}
and this completes the proof.
\end{proof}

Note that if $M$ is a Sasakian manifold, that is, if $s=1$, $\alpha=1$ and $\beta=0$, this method does not produce an $(\alpha,\beta)$ trans-Sasakian manifold but a $(\alpha,0)$ one because, by Darboux's theorem, if $a,b$ only depend of the direction of $\xi$, they should be constants.
\begin{cor}
Let $(M,f,\xi_1,\dots,\xi_s,\eta_1,\dots,\eta_s,g)$ be an $S$-manifold and consider a ge\-ne\-ra\-lized $D$-conformal deformation on $M$, with $a,b$ positive functions depen\-ding only
on the directions of the structure vector fields. Then $(M,\widetilde f,\widetilde\xi_1,\dots,\widetilde\xi_s,\widetilde\eta_1,\dots,\widetilde\eta_s,\widetilde g)$ is a trans-$S$-manifold with functions:
\begin{equation*}
\widetilde\alpha_i=\fra{a}{b},\mbox{  }\widetilde\beta_i=\fra{\xi_ib}{2ab},\mbox{    }i=1,\dots,s.
\end{equation*}
\end{cor}
\begin{cor}
Let $(M,f,\xi_1,\dots,\xi_s,\eta_1,\dots,\eta_s,g)$ be an $C$-manifold and consider a ge\-ne\-ra\-lized $D$-conformal deformation on $M$, with $a,b$ positive functions depen\-ding only
on the directions of the structure vector fields. Then $(M,\widetilde f,\widetilde\xi_1,\dots,\widetilde\xi_s,\widetilde\eta_1,\dots,\widetilde\eta_s,\widetilde g)$ is a trans-$S$-manifold with functions:
\begin{equation*}
\widetilde\alpha_i=0,\mbox{    }\widetilde\beta_i=\fra{\xi_ib}{2ab},\mbox{    }i=1,\dots,s.
\end{equation*}
\end{cor}
\begin{cor}
Let $(M,f,\xi_1,\dots,\xi_s,\eta_1,\dots,\eta_s,g)$ be a generalized Kenmotsu manifold and consider a ge\-ne\-ra\-lized $D$-conformal deformation on $M$, with $a,b$ positive functions depen\-ding only
on the directions of the structure vector fields. Then $(M,\widetilde f,\widetilde\xi_1,\dots,\widetilde\xi_s,\widetilde\eta_1,\dots,\widetilde\eta_s,\widetilde g)$ is a trans-$S$-manifold with functions:
\begin{equation*}
\widetilde\alpha_i=0,\mbox{    }\widetilde\beta_i=\frac{\xi_ib}{2ab}+\frac{1}{a},\mbox{    }i=1,\dots,s.
\end{equation*}
\end{cor}

Next, we are going to construct more examples of trans-$S$-manifolds by using warped products. For later use, we need the following lemma from \cite{oneill} to compute the Riemannian connection of a warped product:
\begin{lem}\label{conexwp}
Let us consider $M=B\times _{h} F$ and denote by $\nabla$, $\nabla^{B}$ and $\nabla^{F}$ the Riemannian connections on $M$, $B$ and $F$. If $X,Y$ are tangent vector fields on $B$ and $V,W$ are tangent vector fields on $F$, then:
\begin{enumerate}
\item[(i)]$\nabla_{X}Y$ is the lift of $\nabla^{B}_{X}Y$.
\item[(ii)]$\nabla_{X}V=\nabla_{V}X=(Xh/h)V.$
\item[(iii)]The component
of $\nabla_{V}W$ normal to the fibers is:
$$-(g_{h}(V,W)/h)\mbox{\rm grad } h.$$
\item[(iv)]The component of $\nabla_{V}W$ tangent to the fibers is the lift of $\nabla^{F} _{V}W$.
\end{enumerate}
\end{lem}

In this context, given an almost Hermitian manifold $(N,J,G)$, the warped product $\w M=\R^s\times_{h}N$ can be endowed with a metric $f$-structure $(\w f,\w\xi_1,\dots,\w\xi_s,\w\eta_1,\dots,\w\eta_s,g_h)$, with the warped metric
\begin{equation*} \label{metwpRW}
g_{h}=-\pi^{*}(g_{\R^s})+(h\circ\pi )^{2}\sigma ^{*}(G),
\end{equation*}
where $h>0$ is a differentiable function on $\R^s$ and $\pi$ and $\sigma$ are the
projections from $\R^s\times N$ on $\R^s$ and $N$, respectively.  In fact, $\w f(\w X)=(J\sigma _{*}\w X)^{*}$, for any vector field $\w X\in\mathcal{X}(\w M)$ and $\w\xi_i=\partial/\partial t_i$, $i=1,\dots,s$, where $t_i$ denotes the coordinates of $\R^s$. Note that this metric is the one used to construct the
Robertson-Walker spaces (see \cite{oneill}).

Now, we study the structure of this warped product.
\begin{theorem} \label{wpestructura}
Let $N$ be an almost Hermitian manifold. Then, the warped product $(\w M=\R^s\times_h N,\w f,\w\xi_1,\dots,\w\xi_s,\w\eta_1,\dots,\w\eta_s,g_h)$ is a trans-$S$-man\-i\-fold
with functions $\w\alpha_1=\cdots=\w\alpha_s=0$ and $\w\beta_i={h^{i)}}/{h}$, $i=1,\dots,s$, if and only if $N$ is a Kaehlerian manifold, where $h^{i)}$ are denoting the components of the gradient of the function $h$, for $i=1,\dots,s$.
\end{theorem}
\begin{proof} Consider $\w X=U+X$ and $\w Y=V+Y$, where $U,V$ and $X,Y$ are tangent vector fields on $\R^s$ and $N$, respectively. Then, taking into account Lemma \ref{conexwp}, if $\w\nabla$ and $\nabla^N$ denote the Riemannian connections of $\w M$ and $N$, respectively, we have:
\begin{equation*}
\begin{split}
(\w\nabla_{\w X}\w f)\w Y=&\w\nabla_{U}JX+\nabla_XJY\\
&-\w f(\w\nabla_{U}V+\w\nabla_XV+\w\nabla_{U}Y+\w\nabla_XY)\\
=&\frac{U(h)}{h}JY-\frac{g_h(X,JY)}{h}\grad(h)+\nabla_{X}^{N}JY\\ &-f(\nabla_UV+\frac{V(h)}{h}X+\frac{U(h)}{h}Y-\frac{g_h(X,Y)}{h}\grad(h)+\nabla_XY)\\
=&-\frac{g_h(X,JY)}{h}\grad(h)-\frac{V(h)}{h}JX+(\nabla_X^NJ)Y\\
=&\fra{g_h(JX,Y)}{h}\sum_{i=1}^sh^{i)}\w\xi_i-\sum_{i=1}^s \w\eta_i(V)\frac{h^{i)}}{h}JX+(\nabla_X^NJ)Y\\
=&\fra{g_h(\w f\w X,\w Y)}{h}\sum_{i=1}^sh^{i)}\w\xi_i-\sum_{i=1}^s
\w\eta_i(V)\frac{h^{i)}}{h}\w f\w X+(\nabla_X^NJ)Y.
\end{split}
\end{equation*}

Therefore, (\ref{nabla}) holds if and only if $(\nabla_{U}^{N}J)V=0$, that is, if and only if $N$ is a Kaehlerian manifold. Moreover, for any $i=1,\dots,s$,
$$\w\nabla_{\w X}\w\xi_i=\nabla_{U}\w\xi_i+\nabla_{X}\w\xi_i=\fra{h^{i)}}{h}X=\fra{h^{i)}}{h}(\w X-\sum_{i=1}^{s}\w\eta(\w X)\w\xi_i)=-\frac{h^{i)}}{h}\w f^2\w X$$
and then, Theorem \ref{teonablaxi} gives the result.
\end{proof}
\begin{cor} The warped product $\R^s\times_h N$, being $N$ a Kaehlerian manifold and $h$ a constant function, is a $C$-manifold. In particular, if $h=1$, the Riemannian product $\R^s\times N$ is a $C$-manifold.
\end{cor}

Combining these examples with a generalized $D$-conformal deformation, a great variety of non-trivial trans-$S$-manifolds can be presented.

Moreover, if we do the warped product of $\R^s$ with a $(2n+s_1)$-dimensional (almost) trans-$S$-manifold $(M,f,\xi_1,\dots,\xi_{s_1},\eta_1,\dots,\eta_{s_1},g)$, we obtain a new metric $f$-manifold
$$(\w M=\R^s\times_h M, \w f,\w\xi_1,\dots,\w\xi_{s+s_1},\w\eta_1,\dots,\w\eta_{s+s_1},g_h),$$
where $\w f(\w X)=(f\sigma_*\w X)^*$ and:
$$\w\xi_i=\left\{\begin{array}{lcl}
\displaystyle{\frac{\partial}{\partial t_i}} & \mbox{\rm if} & 1\leq i\leq s,\\
\\
\displaystyle{\frac{1}{h}\xi_{i-s}} & \mbox{\rm if} & s+1\leq i\leq s+s_1.
\end{array}
\right.$$

These manifolds, under certain hypothesis about the function $h$, verify (\ref{nabla}) but not (\ref{ts2}), so from Theorem \ref{teonablaxi} they are not normal. Consequently, they are examples of almost trans-$S$-manifolds not trans-$S$-manifolds.

\begin{theorem}
Let $M$ be a $(2n+s_1)$-dimensional (almost) trans-S-manifold with functions $(\alpha_i,\beta_i)$, $i=1,\dots,s_1$. Then, the warped product $\w M=\R^s\times_hM$, with the metric $f$-structure defined above, is a $(2n+s+s_1)$-dimensional almost trans-S-manifold with fuctions
$$\w\alpha_i=\left\{\begin{array}{ccl}
0 & \mbox{\rm for} & i=1,\dots,s,\\
\\
\displaystyle{\frac{\alpha_{i-s}}{h}} & \mbox{\rm for} & i=s+1,\dots,s+s_1.
\end{array}
\right.$$
and:
$$\w\beta_i=\left\{\begin{array}{ccl}
\displaystyle{\frac{h^{i)}}{h}} & \mbox{\rm for} & i=1,\dots,s,\\
\\
\displaystyle{\frac{\beta_{i-s}}{h}} & \mbox{\rm for} & i=s+1\dots,s+s_1.
\end{array}
\right.$$
\end{theorem}
\begin{proof} Consider $\w X=U+X$ and $\w Y=V+Y$, where $U,V$ and $X,Y$ are tangent vector fields on $\R^s$ and $M$, respectively. Then, taking into account Lemma \ref{conexwp}, if $\nabla$ is the Riemannian connection of $M$, we deduce:
\begin{equation*}
\begin{split}
(\nabla_{\w X}\w f)\w Y=&-\frac{g_h(X,fY)}{h}\grad(h)-\frac{V(h)}{h}fX+(\nabla_Xf)Y\\
=&\fra{g_h(fX,Y)}{h}\sum_{i=1}^sh^{i)}\frac{\partial}{\partial t_i}-\sum_{i=1}^sV(t_i)\frac{h^{i)}}{h}fX\\
&+\sum_{i=s+1}^{s+s_1}\left[\alpha_{i-s}\left\{g(fX,fY)\xi_{i-s}+\eta_{i-s}(Y)f^2X\right\}\right.\\
&+\left.\beta_{i-s}\left\{g(fX,Y)\xi_{i-s}-\eta_{i-s}(Y)fX\right\}\right]\\
=&\sum_{i=1}^s\fra{h^{i)}}{h}\{g_h(\w f\w X, \w Y)\w\xi_i-\w\eta_{i}(\w Y)\w f\w X\}\\
&+\sum_{i=s+1}^{s+s_1}\left[\frac{\alpha_{i-s}}{h}\{g_h(\w f\w X,\w f\w Y)\w\xi_{i-s}+\w\eta_{i-s}(\w Y)\w f^2\w X\}\right.\\
&+\left.\frac{\beta_{i-s}}{h}\{g_h(\w f\w X,\w Y)\w\xi_{i-s}-\w\eta_{i-s}(\w Y)\w f\w X\}\right].
\end{split}
\end{equation*}

Joining the addends appropriately, it takes the form of (\ref{nabla}) with the desired functions. Therefore, $\w M$ is a almost trans-$S$-manifold.
\end{proof}

Observe that, in the above conditions, (\ref{ts2}) is not verified in general. In fact, consider $\w\xi_i$ with $1\leq i\leq s$. Then, for any $\w X\in\mathcal{X}(\w M)$,
$$\w\nabla_{\w X}\w\xi_i=\fra{h^{i)}}{h}U=\fra{h^{i)}}{h}(\w X-\sum_{j=1}^{s}\w\eta_j(\w X)\xi_j)$$
and so, if $h$ is not a constant function, from Theorem \ref{teonablaxi}, we get that $\w M$ is not a trans-$S$-manifold.
\begin{cor} The warped product $\R^s\times_h M$, being $M$ a trans-$S$-manifold, is a trans-$S$-manifold if and only if $h$ is constant. In particular, the Riemannian product $\R^s\times M$ is a trans-$S$-manifold with functions
$$(0,\stackrel{s)}\dots,0,\alpha_1,\dots,\alpha_{s_1},0,\stackrel{s)}\dots,0,\beta_1,\dots,\beta_{s_1}),$$
where $(\alpha_i,\beta_i)$, $i=1,\dots,s_1$, denote the characteristic functions of $M$.
\end{cor}
\begin{cor}
Let $M$ be a Sasakian manifold. Then, the warped product $\R\times_h M$ is a almost trans-$S$-manifold with functions:
$$\alpha_1=0,\mbox{ }\alpha_2=\fra{1}{h},\mbox{ }\beta_1=\frac{h'}{h}\mbox{ and }\beta_2=0.$$
\end{cor}
\begin{cor}
Let $M$ be a three dimensional trans-Sasakian, with non-constant characteristic functions $\alpha$ and $\beta$. Then, the warped product $\R\times_h M$ is a four dimensional almost trans-$S$-manifold not trans-$S$-manifold with functions:
$$\alpha_1=0,\mbox{ }\alpha_2=\fra{\alpha}{h},\mbox{ }\beta_1=\frac{h'}{h}\mbox{ and }\beta_2=\fra{\beta}{h}.$$
\end{cor}

\end{document}